\documentclass[12pt, twoside, leqno]{article}
\usepackage{amsmath,amsthm}
\usepackage{amssymb,latexsym}
\usepackage{enumerate}

\newtheorem{thm}{Theorem}[section]

\newtheorem{lem}[thm]{Lemma}

\newtheorem{mainthm}[thm]{Main Theorem}

\theoremstyle{definition}
\newtheorem{defin}[thm]{Definition}

\numberwithin{equation}{section}

\def\qed{\vbox{\hrule
 \hbox{\vrule\hbox to 5pt{\vbox to 8pt{\vfil}\hfil}\vrule}\hrule}}

\def\({\left(}
\def\){\right)}

\newcommand{\labt}[1]{\label{thm:#1}}

\newcommand{\labmt}[1]{\label{mthm:#1}}
\newcommand{\refmt}[1]{Main Theorem~\ref{mthm:#1}}
\newcommand{\labl}[1]{\label{lemma:#1}}
\newcommand{\refl}[1]{Lemma~\ref{lemma:#1}}

\newcommand{\labd}[1]{\label{definition:#1}}
\newcommand{\refd}[1]{Definition~\ref{definition:#1}}

\begin{document}


\baselineskip=17pt


\title{Construction of normal numbers with respect to the $Q$-Cantor series expansion for certain $Q$}

\author{Bill Mance\\
Department of Mathematics, The Ohio State University\\
231 West 18th Avenue\\
Columbus, OH 43210-1174\\
E-mail: mance@math.ohio-state.edu}

\date{}

\maketitle


\renewcommand{\thefootnote}{}



\renewcommand{\thefootnote}{\arabic{footnote}}
\setcounter{footnote}{0}


\begin{abstract}
A. R\'enyi \cite{Renyi} made a definition that gives one generalization of simple normality in the context of $Q$-Cantor series.  Similarly, in this paper we give a definition which generalizes the notion of normality in the context of $Q$-Cantor series.  We will prove a theorem that allows us to concatenate sequences of digits that have a special property to give us the digits of a $Q$-normal number for certain $Q$.  We will then use this theorem to construct a Q and a real number $x$ that is $Q$-normal.
\end{abstract}

\section{Introduction}

\begin{defin}\labd{def1.1}
A {\it block of length $k$ in base $b$} is an ordered $k$-tuple of integers in $\{0,1,\ldots,b-1\}$.  A {\it block of length $k$} will be understood to be a block of length $k$ in some base $b$.  A {\it block} will mean a block of length $k$ in base $b$ for some integers $k$ and $b$.
\end{defin}

Given a block $B$, $|B|$ will represent the length of $B$.  Given blocks $B_1,B_2,\ldots,B_n$ and integers $l_1,l_2,\ldots,l_n$, the block
\begin{equation}
B=l_1B_1 l_2B_2 \ldots l_n B_n
\end{equation}
will be the block of length $l_1 |B_1|+\ldots+l_n |B_n|$ formed by concatenating $l_1$ copies of $B_1$, $l_2$ copies of $B_2$, all the way up to $l_n$ copies of $B_n$.  For example, if $B_1=(2,3,5)$ and $B_2=(0,8)$ then $2B_1 1B_2=(2,3,5,2,3,5,0,8)$.

\begin{defin}\labd{def1.2} Given an integer $b \geq 2$, the {\it $b$-ary expansion} of a real $x$ in $[0,1)$ will be the (unique) expansion of the form
\begin{equation} \label{eqn:bary} 
x=\sum_{n=1}^{\infty} \frac {E_n} {b^n}=0.E_1 E_2 E_3 \ldots
\end{equation}
such that all $E_n$ can take on the values $0,1,\ldots,b-1$ with $E_n \neq b-1$ infinitely often.
\end{defin}

We will let $N_n^b(B,x)$ denote the number of times a block $B$ occurs with starting position no greater than $n$ in the $b$-ary expansion of $x$.

\begin{defin}\labd{def1.3} A real number $x$ in $[0,1)$ is {\it normal in base $b$} if for all $k$ and blocks $B$ in base $b$ of length $k$,
\begin{equation} \label{eqn:bnormal1}
\lim_{n \rightarrow \infty} \frac {N_n^{b}(B,x)} {n}=b^{-k}.
\end{equation}
A number is {\it simply normal in base $b$} if $(\ref{eqn:bnormal1})$ holds for $k=1$.
\end{defin}

 Borel introduced normal numbers in 1909 and proved that Lebesgue almost every real number in $[0,1)$ is simultaneously normal to all bases.  The best known example of a number normal in base $10$ is due to Champernowne \cite{Champernowne}.  The number
\begin{equation*}
H_{10}=0.1 \ 2 \ 3 \ 4 \ 5 \ 6 \ 7 \ 8 \ 9 \ 10  \ 11 \ 12 \ldots ,
\end{equation*}
formed by concatenating the digits of every natural number written in increasing order in base $10$, is normal in base $10$.  Any $H_b$, formed similarly to $H_{10}$ but in base $b$, is known to be normal in base $b$. There have since been many examples given of numbers that are normal in at least one base.  One  can find a more thorough literature review in \cite{DT} and \cite{KuN}.

The $Q$-Cantor series expansion, first studied by Georg Cantor, is a natural generalization of the $b$-ary expansion.

\begin{defin}\labd{def1.4}
$Q=\{q_n\}_{n=1}^{\infty}$ is a {\it basic sequence} if each $q_n$ is an integer greater than or equal to $2$.
\end{defin}

\begin{defin}\labd{def1.5}
Given a basic sequence $Q$, the {\it $Q$-Cantor series expansion} of a real $x$ in $[0,1)$ is the (unique) expansion of the form

\begin{equation} \label{eqn:cseries}
x=\sum_{n=1}^{\infty} \frac {E_n} {q_1 q_2 \ldots q_n}
\end{equation}
such that $E_n$ can take on the values $0,1,\ldots,q_n-1$ with $E_n \neq q_n-1$ infinitely often.\footnote{Uniqueness can be proven in the same way as for the $b$-ary expansion.}
\end{defin}

Clearly, the $b$-ary expansion is a special case of $(\ref{eqn:cseries})$ where $q_n=b$ for all $n$.  If one thinks of a $b$-ary expansion as representing an outcome of repeatedly rolling a fair $b$-sided die, then a $Q$-Cantor series expansion may be thought of as representing an outcome of rolling a fair $q_1$ sided die, followed by a fair $q_2$ sided die and so on.  For example, if $q_n=n+1$ for all $n$ then the $Q$-Cantor series expansion of $e-2$ is
\begin{equation*}
e-2=\frac{1} {2}+\frac{1} {2 \cdot 3}+\frac{1} {2 \cdot 3 \cdot 4}+\ldots
\end{equation*}
If $q_n=10$ for all $n$, then the $Q$-Cantor series expansion for $1/4$ is
\begin{equation*}
\frac {1} {4}=\frac{2} {10}+\frac {5} {10^2}+\frac {0} {10^3}+\frac {0} {10^4}+\ldots
\end{equation*}

For a given basic sequence $Q$, let $N_n^Q(B,x)$ denote the number of times a block $B$ occurs starting at a position no greater than $n$ in the $Q$-Cantor series expansion of $x$. Additionally, define
\begin{equation} \label{eqn:Qnk}
Q_n^{(k)}=\sum_{j=1}^n \frac {1} {q_j q_{j+1} \ldots q_{j+k-1}}.
\end{equation}

A. R\'enyi\cite{Renyi} defined a real number $x$ to be normal if for all blocks $B$ of length $1$,
\begin{equation}\label{eqn:rnormal}
\lim_{n \rightarrow \infty} \frac {N_n^Q (B,x)} {Q_n^{(1)}}=1.
\end{equation}
If $q_n=b$ for all $n$ then $(\ref{eqn:rnormal})$ is equivalent to simply normal in base $b$, but not equivalent to normal in base $b$.  Thus, we wish to generalize normality in a way that will be equivalent to normality in base $b$ when all $q_n=b$.

\begin{defin}\labd{def1.6}
A basic sequence $Q$ is {\it infinite limit} if $q_n \rightarrow \infty$.
\end{defin}

\begin{defin}\labd{def1.7}
A real number $x$ is {\it $Q$-normal of order $k$} if for all blocks $B$ of length $k$,
\begin{equation}
\lim_{n \rightarrow \infty} \frac {N_n^Q (B,x)} {Q_n^{(k)}}=1.
\end{equation}
$x$ is said to be {\it $Q$-normal} if it is $Q$-normal of order $k$ for all $k$.
\end{defin}

\begin{defin}\labd{def1.8}
A basic sequence $Q$ is {\it $k$-divergent} if
\begin{equation}
\lim_{n \rightarrow \infty} Q_n^{(k)}=\infty.
\end{equation}
$Q$ is {\it fully divergent} if $Q$ is $k$-divergent for all $k$.
\end{defin}

It has been shown that for infinite limit $Q$, the set of all $x$ in $[0,1)$ that are $Q$-normal of order $k$ has full Lebesgue measure if and only if $Q$ is $k$-divergent \cite{Renyi}.  Therefore, for infinite limit $Q$, the set of  all $x$ in $[0,1)$ that are $Q$-normal has full Lebesgue measure if and only if $Q$ is fully divergent.  Similarly to the case of the $b$-ary expansion, it will be more difficult to construct specific examples of $Q$-normal numbers than to show the typical real number is $Q$-normal.

The situation is further complicated when $Q$ is infinite limit because in that case we need to consider blocks whose digits come from an infinite set.  For example, normality can be defined for the continued fraction expansion.  In that setup there will also be an infinite digit set.  While it is known that almost every real number is normal with respect to the continued fraction expansion, there are not many known examples (see \cite{AKS} and \cite{PostPyat}).

We wish to state a theorem that will allow us to construct specific examples of $Q$-normal numbers for certain $Q$.  We will first need several definitions.

\begin{defin}\labd{def1.9}
\footnote{\cite{Post} discusses normality in base $2$ with respect to different weightings.} A {\it weighting} $\mu$ is a collection of functions $\mu^{(1)},\mu^{(2)},\mu^{(3)},\ldots$ such that for all $k$
\begin{equation}
\mu^{(k)}:\{0,1,2,\ldots\}^k \rightarrow [0,1];
\end{equation}
\begin{equation}
\sum_{j=0}^{\infty} \mu^{(1)}(j)=1;
\end{equation}
\begin{equation}
\mu^{(k)}(b_1,b_2,\ldots,b_k)=\sum_{j=1}^{\infty} \mu^{(k+1)}(b_1,b_2,\ldots,b_k,j).
\end{equation}
\end{defin}

\begin{defin}\labd{def1.10}
The {\it uniform weighting in base $b$} is the collection $\lambda_b$ of functions  $\lambda_b^{(1)},\lambda_b^{(2)},\lambda_b^{(3)},\ldots$ such that for all $k$ and blocks $B$ of length $k$ in base $b$
\begin{equation}
\lambda_b^{(k)}(B)=b^{-k}.
\end{equation}
\end{defin}

\begin{defin}\labd{def1.11}
Let $p$ and $b$ be positive integers such that $1 \leq p \leq b$.  A weighting $\mu$ is {\it $(p,b)$-uniform} if for all $k$ and blocks $B$ of length $k$ in base $p$, we have
\begin{equation}
\mu^{(k)}(B)=\lambda_b^{(k)}(B)=b^{-k}.
\end{equation}
\end{defin}

Given blocks $B$ and $y$ we will let $N_n(B,y)$ denote the number of times a block $B$ occurs starting in position no greater than $n$ in the block $y$.

\begin{defin}\labd{def1.12}
Suppose that $0<\epsilon < 1$, $k$ is a positive integer and $\mu$ is a weighting.  A block of digits $y$ is {\it $(\epsilon,k,\mu)$-normal }\footnote{\refd{def1.12} is a generalization of the concept of $(\epsilon,k)$-normality, originally due to Besicovitch \cite{Besicovitch}.} if for all blocks $B$ of length $m \leq k$, we have
\begin{equation}
\mu^{(m)}(B)|y|(1-\epsilon) \leq N_{|y|}(B,y) \leq \mu^{(m)}(B)|y|(1+\epsilon).
\end{equation}
\end{defin}

For convenience, we define the notion of a block friendly family (BFF):

\begin{defin}\labd{def1.13}
A {\it BFF} is a sequence of $6$-tuples $W=\{(l_i,b_i,p_i,\epsilon_i,k_i,\mu_i)\}_{i=1}^{\infty}$ with non-decreasing sequences of non-negative integers $\{l_i\}_{i=1}^{\infty}$, $\{b_i\}_{i=1}^{\infty}$, $\{p_i\}_{i=1}^{\infty}$ and $\{k_i\}_{i=1}^{\infty}$ for which $b_i \geq 2$, $b_i \rightarrow \infty$ and $p_i \rightarrow \infty$, such that $\{\mu_i\}_{i=1}^{\infty}$ is a sequence of $(p_i,b_i)$-uniform weightings and $\{\epsilon_i\}_{i=1}^{\infty}$ strictly decreases to $0$.
\end{defin}

We will use the notation
\begin{equation*}
f(n)=\omega(g(n))
\end{equation*}
to mean that $f$ asymptotically dominates $g$.  In other words,
\begin{equation*}
\lim_{n \rightarrow \infty} \frac {f(n)} {g(n)} = \infty.
\end{equation*}

\begin{defin}\labd{def1.14}
Let $W=\{(l_i,b_i,p_i,\epsilon_i,k_i,\mu_i)\}_{i=1}^{\infty}$ be a BFF.  If $\lim k_i=K<\infty$, then let $R(W)=\{0,1,2,\ldots,K\}$. Otherwise, let $R(W)=\{0,1,2,\ldots\}$. If $\{x_i\}_{i=1}^{\infty}$ is a sequence of blocks such that $|x_i|$ is non-decreasing and $x_i$ is $(\epsilon_i,k_i,\mu_i)$-normal, then $\{x_i\}_{i=1}^{\infty}$ is said to be {\it $W$-good} if for all $k$ in $R$,
\begin{equation}\label{eqn:good1}
|x_i| = \omega \left(\frac {b_i^k} {\epsilon_{i-1}-\epsilon_i} \right);
\end{equation}
\begin{equation}\label{eqn:good2}
\frac {l_{i-1}} {l_i} \cdot \frac {|x_{i-1}|} {|x_i|}=o(i^{-1}b_i^{-k});
\end{equation}
\begin{equation}\label{eqn:good3}
\frac {1} {l_i} \cdot \frac {|x_{i+1}|} {|x_i|}=o(b_i^{-k}).
\end{equation}
\end{defin}

For the rest of the paper, given a BFF $W$ and a $W$-good sequence $\{x_i\}$, we will define
\begin{equation}
L_i=|l_1 x_1 \ldots l_i x_i|=\sum_{j=1}^i l_j |x_j|=l_1 |x_1|+\ldots+l_i |x_i|,
\end{equation}
\begin{equation}
q_n=b_i \textrm{\ for $L_{i-1} < n \leq L_i$},
\end{equation}
and
\begin{equation}
Q=\{q_n\}_{n=1}^{\infty}.
\end{equation}
Moreover, if $(E_1,E_2,\ldots)=l_1x_1 l_2 x_2 \ldots$ then let
\begin{equation}
x=\sum_{n=1}^{\infty} \frac {E_n} {q_1 q_2 \ldots q_n}.
\end{equation}

With these conventions, we are now in a position to state \refmt{thm3.1}.

\begin{mainthm}\labmt{thm3.1}
Let $W$ be a BFF and $\{x_i\}_{i=1}^{\infty}$ a $W$-good sequence. If $k \in R(W)$, then $x$ is $Q$-normal of order $k$.  If $k_i \rightarrow \infty$, then $x$ is $Q$-normal.
\end{mainthm}

Let $C_{b,w}$ be the block formed by concatenating all the blocks of length $w$ in base $b$ in lexicographic order.  For example,
$$
C_{3,2}=1(0,0)1(0,1)1(0,2)1(1,0)1(1,1)1(1,2)1(2,0)1(2,1)1(2,2)
$$
$$
=(0,0,0,1,0,2,1,0,1,1,1,2,2,0,2,1,2,2).
$$
Let $x_1=(0)$, $b_1=2$ and $l_1=0$.  For $i \geq 2$, let $x_i=C_{i,i^2}$, $b_i=i$ and $l_i=i^{3i}$.  We will show in section 4 that $x$ is $Q$-normal.\footnote{This result will not require the full generality of $(p,b)$-uniform weightings considered in \refmt{thm3.1}, but they will be required in a later paper.}

\section{Technical Lemmas}

For this section, we will fix a BFF $W$ and a $W$-good sequence $\{x_i\}$.  For a given $n$, the letter $i=i(n)$ will always be understood to be the positive integer that satisfies $L_{i-1} < n \leq L_i$.  This usage of $i$ will be made frequently and without comment.  Let $m=n-L_i$, which allows $m$ to be written in the form
\begin{equation*}
m=\alpha |x_{i+1}|+\beta
\end{equation*}
where $\alpha$ and $\beta$ satisfy
\begin{equation*}
0 \leq \alpha \leq l_{i+1} \textrm{ and } 0 \leq \beta < |x_{i+1}|.
\end{equation*}

Thus, we can write the first $n$ digits of $x$ in the form

\begin{equation}
l_1 x_1 l_2 x_2 \ldots l_{i-1} x_{i-1} \ l_i x_i \ \alpha x_{i+1} \ 1y, 
\end{equation}
where $y$ is the block formed from the first $\beta$ digits of $x_{i+1}$.

Given a block $B$ of length $k$ in $R(W)$, we will first get upper and lower bounds on $N_n^Q(B,x)$, which will hold for all $n$ large enough that $k \leq k_i$.  This will allow us to bound 

\begin{equation}\label{eqn:bnd}
\left| \frac {N_n^Q(B,x)} {Q_n^{(k)}} -1 \right|
\end{equation}
and show that 
\begin{equation}
\lim_{n \rightarrow \infty} \frac {N_n^Q(B,x)} {Q_n^{(k)}}=1.
\end{equation}

We will arrive at upper and lower bounds for $N_n^Q(B,x)$ by breaking the first $n$ digits of $x$ into three parts: the initial block $l_1 x_1 l_2 x_2 \ldots l_{i-1} x_{i-1}$, the middle block $l_i x_i$ and the last block $\alpha x_{i+1} \ 1y$.

\begin{lem}\labl{l2.1}
If $k \leq k_i$ and $B$ is a block of length $k$ in base $b \leq p_i$, then the following bounds hold:
\begin{equation}
(1-\epsilon_i)b_i^{-k}|x_i| \leq N_{|x_i|}(B,x_i) \leq (1+\epsilon_i)b_i^{-k}|x_i|;
\end{equation}
\begin{equation}
(1-\epsilon_{i+1})b_{i+1}^{-k} \alpha |x_{i+1}| \leq N_m(B,l_{i+1}x_{i+1})  \leq (1+\epsilon_{i+1})b_{i+1}^{-k} \alpha |x_{i+1}|+\beta+k\alpha.
\end{equation}
\end{lem}
\begin{proof}
Since $x_i$ is $(\epsilon_i,k_i,\mu_i)$-normal and $\mu_i$ is $(p_i,b_i)$-uniform, it immediately follows that
\begin{equation*}
(1-\epsilon_i)b_i^{-k}|x_i| \leq N_{|x_i|}(B,x_i) \leq (1+\epsilon_i)b_i^{-k}|x_i|.
\end{equation*}

We can estimate $N_m(B,l_{i+1}x_{i+1})$ by using the fact that $k \leq k_{i+1}$ and $x_{i+1}$ is $(\epsilon_{i+1},k_{i+1},\mu_{i+1})$-normal so that
\begin{equation*}
(1-\epsilon_{i+1})b_{i+1}^{-k}|x_{i+1}| \leq N_{|x_{i+1}|}(B,x_{i+1}) \leq (1+\epsilon_{i+1})b_{i+1}^{-k}|x_{i+1}|.
\end{equation*}
The upper bound for $N_m(B,l_{i+1}x_{i+1})$ is determined by assuming that $B$ occurs at every location in the initial substring of length $\beta$ of a copy of $x_{i+1}$ and $k$ times on each of the $\alpha$ boundaries.  The lower bound is attained by assuming $B$ never occurs in these positions, so
\begin{equation*}
(1-\epsilon_{i+1})b_{i+1}^{-k} \alpha |x_{i+1}| \leq N_m(B,l_{i+1}x_{i+1})  \leq (1+\epsilon_{i+1})b_{i+1}^{-k} \alpha |x_{i+1}|+\beta+k\alpha. 
\end{equation*}
\end{proof}

We define the following quantity, which simplifies the statement of \refl{l2.2} and proof of \refl{l2.4}:
$$
\kappa=\left( L_{i-1}+k(l_i+1)+(1+\epsilon_i)b_i^{-k} l_i |x_i| \right)+((1+\epsilon_{i+1})b_{i+1}^{-k} |x_{i+1}|+k)\alpha+\beta.
$$
\begin{lem}\labl{l2.2}
If $k \leq k_i$ and $B$ is a block of length $k$ in base $b \leq p_i$, then
\begin{equation}
(1-\epsilon_i)b_i^{-k} l_i |x_i|+(1-\epsilon_{i+1})b_{i+1}^{-k} \alpha |x_{i+1}| \leq N_n^Q(B,x) \leq \kappa.
\end{equation}
\end{lem}
\begin{proof} 
For the lower bound, we consider the case where $B$ never occurs in any of the blocks $x_j$ or on the borders for $j < i$.  By combining this with our estimates for $N_{|x_i|}(B,x_i)$ and $N_m(B,l_{i+1}x_{i+1})$ in \refl{l2.1}, we get
\begin{equation*}
N_n^Q(B,x) \geq (1-\epsilon_i)b_i^{-k} l_i |x_i|+(1-\epsilon_{i+1})b_{i+1}^{-k} \alpha |x_{i+1}|.
\end{equation*}

Next, we can get an upper bound for $N_n^Q(B,x)$.  Here we assume that $B$ occurs at every position in each of the $x_j$ for $j<i$ and $k$ times on each of the boundaries.
\begin{equation*}
N_n^Q(B,x) \leq (l_1 |x_1|+\ldots+l_{i-1}|x_{i-1}|)+(1+\epsilon_i)b_i^{-k} l_i |x_i|
\end{equation*}
\begin{equation*}
+(1+\epsilon_{i+1})b_{i+1}^{-k} \alpha |x_{i+1}|+\beta+k(l_i+1+\alpha)
\end{equation*}
\begin{equation*}
=\left( L_{i-1}+k(l_i+1)+(1+\epsilon_i)b_i^{-k} l_i |x_i| \right)+((1+\epsilon_{i+1})b_{i+1}^{-k} |x_{i+1}|+k)\alpha+\beta.
\end{equation*}
\end{proof}

Due to the algebraic complexity of $Q_n^{(k)}$, it will be difficult to directly estimate $(\ref{eqn:bnd})$. Thus, we will introduce a quantity close in value to $Q_n^{(k)}$ that will make this easier.  Let
\begin{equation*}
S_n^{(k)}=\sum_{j=1}^{i} b_j^{-k}l_j |x_j|+b_{i+1}^{-k} m=b_1^{-k}l_1 |x_1|+b_2^{-k}l_2|x_2|+\ldots+b_i^{-k} l_i|x_i|+b_{i+1}^{-k} m.
\end{equation*}

\begin{lem}\labl{l2.3}
$\lim_{n \rightarrow \infty} \frac {Q_n^{(k)}} {S_n^{(k)}}=1.$
\end{lem}
\begin{proof}
Let $s=\min \{t : k<|x_t|\}$.  For $j \geq s$, define
$$
\bar Q_j^{(k)}=\left( \left( \frac {1} {b_j^k}+\ldots+\frac {1} {b_j^k} \right) + \left( \frac {1} {b_j^{k-1} b_{j+1}}+\ldots + \frac {1} {b_j b_{j+1}^{k-1}}\right) \right)
$$
$$
=\frac {l_j |x_j|-(k-1)} {b_j^k} + \sum_{t=1}^{k-1} \frac {1} {b_J^{k-1-t} b_{j+1}^t}.
$$
Thus, by $(\ref{eqn:Qnk})$ and our choice of $Q$, we get that
\begin{equation}
Q_n^{(k)}=Q_{L_{s-1}}^{(k)}+\sum_{j=s}^i \bar Q_j^{(k)}+\sum_{t=L_i+1}^{n} \frac {1} {q_t q_{t+1} \ldots q_{t+k-1}}
\end{equation} 
where the last summation will contain up to $l_{i+1}|x_{i+1}|-(k-1)$ terms identical to $\frac {1} {b_{j+1}^k}$ and up to $k-1$ terms of the form $\frac {1} {b_{i+1}^{k-1-t} b_{i+2}^t}$, depending on $m$.

Similarly to $\bar Q_j^{(k)}$, for $j \geq s$, define
$$
\bar S_j^{(k)}=\left( \frac {1} {b_j^k}+\ldots+\frac {1} {b_j^k} \right)=\frac {l_j |x_j|} {b_j^k}.
$$
Thus,
\begin{equation}
S_n^{(k)}=S_{L_{s-1}}^{(k)}+\sum_{j=s}^i \bar S_j^{(k)}+\sum_{t=L_i+1}^{n} \frac {1} {q_t q_{t+1} \ldots q_{t+k-1}}.
\end{equation} 
We note that almost all terms in $Q_n^{(k)}$ and $S_n^{(k)}$ are identical and are equal to $\frac {1} {b_j^k}$ for some $j$ and will thus cancel out when we consider $S_n^{(k)}-Q_n^{(k)}$.  The only corresponding terms that remain in the difference are thus of the form $\frac {1} {b_j^k}-\frac {1} {b_J^{k-1-t} b_{j+1}^t}$.  However, each of these terms is non-negative as $\{ b_i \}$ is a non-decreasing sequence.  Therefore, $S_n^{(k)}-Q_n^{(k)}$ is non-decreasing in $n$ and 
\begin{equation}\label{eqn:235}
S_n^{(k)} \geq Q_n^{(k)}
\end{equation}
 for all $n$.  In particular, we arrive at the following bound:
\begin{equation}\label{eqn:231} 
S_n^{(k)}-Q_n^{(k)} \leq S_{L_{i+1}}^{(k)}-Q_{L_{i+1}}^{(k)} = \left( S_{L_{s-1}}^{(k)}-Q_{L_{s-1}}^{(k)} \right) + \sum_{j=s}^{i+1} \left( \bar S_j^{(k)}-\bar Q_j^{(k)} \right).
\end{equation}
But,
\begin{equation*}
\bar S_j^{(k)}-\bar Q_j^{(k)}=(l_j |x_j|-(k-1)) \left(\frac {1} {b_j^k}-\frac {1} {b_j^k} \right)+ \sum_{t=1}^{k-1} \left( \frac {1} {b_j^k} - \frac {1} {b_J^{k-1-t} b_{j+1}^t} \right)
\end{equation*}
\begin{equation}\label{eqn:232} 
<(l_j |x_j|-(k-1)) \cdot 0+ \sum_{t=1}^k (1 - 0) = k.
\end{equation}
If we let $r=\left( S_{L_{s-1}}^{(k)}-Q_{L_{s-1}}^{(k)} \right)$ and combine $(\ref{eqn:231})$ and $(\ref{eqn:232})$, then we find that
\begin{equation}\label{eqn:233}
S_n^{(k)}-Q_n^{(k)} < r + \sum_{j=s}^{i+1} k = r + k (i+2-s).
\end{equation}
Lastly, we note that
\begin{equation}\label{eqn:234}
S_n^{(k)}=\sum_{j=1}^{i} b_j^{-k}l_j |x_j|+b_{i+1}^{-k} m \geq l_i |x_i|.
\end{equation}
Using $(\ref{eqn:233})$ and $(\ref{eqn:234})$, we may now show that $\lim_{n \rightarrow \infty} \frac {Q_n^{(k)}} {S_n^{(k)}}=1$:
\begin{equation}
\left| \frac {Q_n^{(k)}} {S_n^{(k)}}-1 \right| = \frac {S_n^{(k)}-Q_n^{(k)}} {S_n^{(k)}} < \frac {(r+k-ks)+ki} {l_i |x_i|}=\frac {r+k-ks} {l_i |x_i|}+k\frac {i} {l_i |x_i|}.
\end{equation}
However, $(r+k-ks)$ is constant with respect to $n$ and $|x_i| \rightarrow \infty$ so $\frac {r+k-ks} {l_i |x_i|} \rightarrow 0$. By $(\ref{eqn:good2})$,  $k\frac {i} {l_i |x_i|} \rightarrow 0$ .
\end{proof}

We will also use the following rational functions, defined on $\mathbb{R}^{\geq 0} \times \mathbb{R}^{\geq 0}$, to estimate $(\ref{eqn:bnd})$:
\begin{equation*}
f_i (w,z)=\frac {\left( S_{L_{i-1}}^{(k)}+\epsilon_i b_i^{-k} l_i |x_i| \right) + \left( \epsilon_{i+1}b_{i+1}^{-k}|x_{i+1}| \right) w +b_{i+1}^{-k} z} {S_{L_i}^{(k)}+(b_{i+1}^{-k} |x_{i+1}|)w+b_{i+1}^{-k}z};
\end{equation*}
\begin{equation*}
g_i(w,z)=\frac{\left( L_{i-1}+\epsilon_i b_i^{-k} l_i |x_i|+k(l_i+1) \right)+(\epsilon_{i+1}b_{i+1}^{-k} |x_{i+1}|+k) w+z} {S_{L_i}^{(k)}+(b_{i+1}^{-k} |x_{i+1}|)w+b_{i+1}^{-k}z}.
\end{equation*}

\begin{lem}\labl{l2.4}
Let $k \in R(W)$ and let $B$ be a block of length $k$ in base $b$.  If $n$ is large enough so that $S_n^{(k)}/Q_n^{(k)}<2$, $k \leq k_i$ and $b \leq p_i$, then
\begin{equation}
\left| \frac {N_n^Q(B,x)} {Q_n^{(k)}} -1 \right| < 2g_i (\alpha,\beta)+\frac {S_n^{(k)}-Q_n^{(k)}} {S_n^{(k)}}.
\end{equation}
\end{lem}
\begin{proof} 
Using our lower bound from \refl{l2.2} on $N_n^Q(B,x)$, $\frac {N_n^Q(B,x)} {Q_n^{(k)}} - 1<0$.  So we use $(\ref{eqn:235})$ and arrive at the upper bound
\begin{equation*}
\left| \frac {N_n^Q(B,x)} {Q_n^{(k)}} -1 \right| \leq 1-\frac {(1-\epsilon_i)b_i^{-k} l_i |x_i|+(1-\epsilon_{i+1})b_{i+1}^{-k} \alpha |x_{i+1}|} {Q_n^{(k)}}
\end{equation*}
\begin{equation*}
<\frac {S_n^{(k)}-((1-\epsilon_i)b_i^{-k} l_i |x_i|+(1-\epsilon_{i+1})b_{i+1}^{-k} \alpha |x_{i+1}|)} {Q_n^{(k)}} \cdot \frac {Q_n^{(k)}} {S_n^{(k)}} \cdot \frac {S_n^{(k)}} {Q_n^{(k)}} 
\end{equation*}
\begin{equation}\label{eqn:241}
<2\frac {S_n^{(k)}-((1-\epsilon_i)b_i^{-k} l_i |x_i|+(1-\epsilon_{i+1})b_{i+1}^{-k} \alpha |x_{i+1}|)} {S_n^{(k)}}=2f_i(\alpha,\beta).
\end{equation}
Similarly to $(\ref{eqn:241})$ and using our upper bound from \refl{l2.2} for $N_n^Q(B,x)$, we can conclude
\begin{equation*}
\left| \frac {N_n^Q(B,x)} {Q_n^{(k)}} - 1 \right| \leq -1+\frac {\kappa} {Q_n^{(k)}} = \frac {\kappa-Q_n^{(k)}} {Q_n^{(k)}}=\frac  {\kappa-S_n^{(k)}} {Q_n^{(k)}} + \frac {S_n^{(k)}-Q_n^{(k)}} {Q_n^{(k)}} 
\end{equation*}
\begin{equation*}
=\frac {\kappa-S_n^{(k)}} {Q_n^{(k)}} \cdot \frac {Q_n^{(k)}} {S_n^{(k)}} \cdot \frac {S_n^{(k)}} {Q_n^{(k)}} + \frac {S_n^{(k)}-Q_n^{(k)}} {Q_n^{(k)}} < 2 \frac {\kappa-S_n^{(k)}} {S_n^{(k)}} + \frac {S_n^{(k)}-Q_n^{(k)}} {Q_n^{(k)}}.
\end{equation*}
But,
\begin{equation*}
\frac {\kappa-S_n^{(k)}} {S_n^{(k)}}=\frac {1} {S_n^{(k)}} \Bigg( \left(\sum_{j=1}^{i-1} (1-j^{-k})l_j |x_j|+k(l_i+1)+\epsilon_i b_i^{-k} l_i |x_i| \right)
\end{equation*}
\begin{equation*}
+(\epsilon_{i+1}b_{i+1}^{-k} |x_{i+1}|+k) \alpha+(1-b_{i+1}^{-k})\beta \Bigg)
\end{equation*}
\begin{equation*}
<\frac{\left( L_{i-1}+\epsilon_i b_i^{-k} l_i |x_i|+k(l_i+1) \right)+(\epsilon_{i+1}b_{i+1}^{-k} |x_{i+1}|+k) \alpha+\beta} {S_{L_i}+(b_{i+1}^{-k} |x_{i+1}|)\alpha+b_{i+1}^{-k}\beta}=g_i(\alpha,\beta).
\end{equation*}
So,
\begin{equation*}
\left| \frac {N_n^Q(B,x)} {Q_n^{(k)}} - 1 \right| <\max \left( 2f_i (\alpha,\beta),2g_i (\alpha,\beta)+\frac {S_n^{(k)}-Q_n^{(k)}} {S_n^{(k)}} \right).
\end{equation*}
However, since the numerator of $g_i (\alpha,\beta)$ is clearly greater than the numerator of $f_i (\alpha,\beta)$ and their denominators are the same we conclude that
\begin{equation*}
f_i (\alpha,\beta)<g_i (\alpha,\beta).
\end{equation*}
Therefore,
\begin{equation*}
\left| \frac {N_n^Q(B,x)} {Q_n^{(k)}} - 1 \right| < 2g_i (\alpha,\beta)+\frac {S_n^{(k)}-Q_n^{(k)}} {S_n^{(k)}}.
\end{equation*}
\end{proof}

In light of \refl{l2.4}, we will want to find a good bound for $g_i(w,z)$ where $(w,z)$ ranges over values in $\{0,1,\ldots,l_{i+1} \} \times \{0,1,\ldots,|x_{i+1}|-1 \}$.

\begin{lem}\labl{l2.5}
If $k \in R(W)$, $|x_i|>4k$, $|x_{i+1}|>\frac {k b_{i+1}^k} {\epsilon_i-\epsilon_{i+1}}$, $l_i>0$ and 
\begin{equation}
(w,z) \in \{0,1,\ldots,l_{i+1} \} \times \{0,1,\ldots,|x_{i+1}|-1 \},
\end{equation}
then
\begin{equation}
g_i(w,z) < g_i (0,|x_{i+1}|)=\frac {(L_{i-1}+\epsilon_i b_i^{-k}l_i|x_i|+k(l_i+1)) +|x_{i+1}|} {S_{L_i}+b_{i+1}^{-k}|x_{i+1}|}.
\end{equation}
\end{lem}
\begin{proof} 
We note that $g_i (w,z)$ is a rational function of $w$ and $z$ of the form
\begin{equation*}
g_i (w,z)=\frac {C+Dw+Ez} {F+Gw+Hz}
\end{equation*}
where
\begin{equation*}
C=L_{i-1}+\epsilon_i b_i^{-k} l_i |x_i|+k(l_i+1), \ D=\epsilon_{i+1}b_{i+1}^{-k} |x_{i+1}|+k, \ E=1,
\end{equation*}
\begin{equation*}
F=S_{L_i}, \ G=b_{i+1}^{-k} |x_{i+1}| \ \textrm{and \ } H=b_{i+1}^{-k}.
\end{equation*}
We will show that if we fix $z$, then $g_i(w,z)$ is a decreasing function of $w$ and if we fix $w$, then $g_i(w,z)$ is an increasing function of $z$.  To see this, we compute the partial derivatives:
\begin{equation*}
\frac {\partial g_i} {\partial w} (w,z)=\frac {D(F+Gw+Hz)-G(C+Dw+Ez)} {(F+Gw+Hz)^2}=\frac {D(F+Hz)-G(C+Ez)} {(F+Gw+Hz)^2};
\end{equation*}
\begin{equation*}
\frac {\partial g_i} {\partial z} (w,z)=\frac {E(F+Gw+Hz)-H(C+Dw+Ez)} {(F+Gw+Hz)^2}=\frac {E(F+Gw)-H(C+Dw)} {(F+Gw+Hz)^2}.
\end{equation*}
Thus, the sign of $\frac {\partial g_i} {\partial w} (w,z)$ does not depend on $w$ and the sign of $\frac {\partial g_i} {\partial z} (w,z)$ does not depend on $z$.  We will first show that $g_i (w,z)$ is an increasing function of $z$ by verifying that
\begin{equation}\label{eqn:251}
E(F+Gw)>H(C+Dw).
\end{equation}
Let
$$
S_i^*=b_{i+1}^{-k}L_{i-1}+\epsilon_i b_i^{-k} b_{i+1}^{-k} l_i |x_i|+b_{i+1}^{-k} k (l_i+1).
$$
Thus, $(\ref{eqn:251})$ can be written as
\begin{equation}\label{eqn:254}
S_{L_i} +  \Bigg[ b_{i+1}^{-k}|x_{i+1}|w \Bigg] > S_i^* + \Bigg[ b_{i+1}^{-k}(\epsilon_{i+1} b_{i+1}^{-k}|x_{i+1}| + k )w \Bigg].
\end{equation}
In order to show that $S_{L_i} > S_i^*$, we first note that
\begin{equation*}
S_{L_i}=S_{L_{i-1}}+b_i^{-k}l_i |x_i|.
\end{equation*}
Since $S_{L_{i-1}} \geq b_{i+1}^{-k}L_{i-1}$, we need to show that
\begin{equation}\label{eqn:252} 
b_i^{-k}l_i |x_i| > b_{i+1}^{-k} (\epsilon_i b_i^{-k}l_i |x_i|+k (l_i+1)).
\end{equation}
However, by rearranging terms, $(\ref{eqn:252})$ is equivalent to
\begin{equation}\label{eqn:253} 
|x_i| > \frac {l_i+1} {l_i} \cdot \left( \frac {b_i} {b_{i+1}} \right)^k \cdot \frac {1} {1-b_{i+1}^{-k}\epsilon_i} \cdot k.
\end{equation}
Since $l_i >0$, we know that $(l_i+1)/l_i \leq 2$.  Since $b_{i+1} \geq 2$ and $\epsilon_i < 1$, we know that $(1-b_{i+1}^{-k}\epsilon_i)^{-1} < 2$.  Additionally, $\{b_i\}$ non-decreasing implies $\left( \frac {b_i} {b_{i+1}} \right)^k \leq 1$. Therefore,
$$
\frac {l_i+1} {l_i} \cdot \left( \frac {b_i} {b_{i+1}} \right)^k \cdot \frac {1} {1-b_{i+1}^{-k}\epsilon_i} \cdot k < 2 \cdot 1 \cdot 2 \cdot k=4k.
$$
But, $|x_i|>4k$. So $(\ref{eqn:253})$ is satisfied and thus $S_{L_i} > S_i^*$.

The last step to verifying $(\ref{eqn:254})$ is to show that
\begin{equation*}
b_{i+1}^{-k}|x_{i+1}|w \geq b_{i+1}^{-k}(\epsilon_{i+1} b_{i+1}^{-k}|x_{i+1}|+k) w.
\end{equation*}
However, this is equivalent to
\begin{equation}\label{eqn:255} 
|x_{i+1}|w \geq (\epsilon_{i+1} b_{i+1}^{-k}|x_{i+1}|+k) w.
\end{equation}
Clearly, $(\ref{eqn:255})$ is true if $w=0$.  If $w>0$ we can cancel out the $w$ term on each side and rewrite $(\ref{eqn:255})$ as
\begin{equation*}
|x_{i+1}| \geq \frac {1} {1-b_{i+1}^{-k}\epsilon_{i+1}} \cdot k.
\end{equation*}
Similar to $(\ref{eqn:253})$, $(1-b_{i+1}^{-k}\epsilon_{i+1})^{-1} k \leq 2 k < |x_i| < |x_{i+1}|$.  Thus $(\ref{eqn:251})$ is satisfied and $g_i(w,z)$ is an increasing function of $z$.

Due to the difficulty of directly showing that $\frac {\partial g_i} {\partial w} (w,z)<0$, we will proceed as follows: because the sign of $\frac {\partial g_i} {\partial w} (w,z)$ does not depend on $w$, we will know that $g_i (w,z)$ is decreasing in $w$ if for each $z$
\begin{equation*}
\lim_{w \rightarrow \infty} g_i (w,z)<g_i(0,z).
\end{equation*}
Since $g_i(w,z)$ is an increasing function of $z$, we know for all $z$ that $g_i(0,0)<g_i(0,z)$.  Hence, it is enough to show that
$$
\lim_{w \rightarrow \infty} g_i (w,z) < g_i(0,0).
$$
Since $\lim_{w \rightarrow \infty} g_i (w,z)=D/G$ and $g_i(0,0)=C/F$, it is sufficient to show that $CG>DF$.  We proceed as follows:

\begin{equation*}
\left( L_{i-1}+\epsilon_i b_i^{-k} l_i |x_i|+k(l_i+1) \right) b_{i+1}^{-k} |x_{i+1}|  
\end{equation*}
\begin{equation*}
> \left( \epsilon_{i+1}b_{i+1}^{-k} |x_{i+1}|+k \right) S_{L_i}=\left( \epsilon_{i+1}b_{i+1}^{-k} |x_{i+1}|+k \right) (S_{L_{i-1}}+b_i^{-k}l_i |x_i|)
\end{equation*}
\begin{equation*}
\Leftrightarrow L_{i-1} b_{i+1}^{-k} |x_{i+1}| + \epsilon_i b_i^{-k} b_{i+1}^{-k} l_i |x_i| |x_{i+1}| + k b_{i+1}^{-k} (l_i+1)  |x_{i+1}|
\end{equation*}
\begin{equation}\label{eqn:256}
>\left( \epsilon_{i+1}b_{i+1}^{-k} |x_{i+1}|+k \right) S_{L_{i-1}} + \left( \epsilon_{i+1}b_{i+1}^{-k} |x_{i+1}|+k \right)b_i^{-k}l_i |x_i|.
\end{equation}
We will verify $(\ref{eqn:256})$ by showing that
\begin{equation}\label{eqn:257}
L_{i-1} b_{i+1}^{-k} |x_{i+1}|>\left( \epsilon_{i+1}b_{i+1}^{-k} |x_{i+1}|+k \right) S_{L_{i+1}} \textrm{ \ and}
\end{equation}
\begin{equation}
\epsilon_i b_i^{-k} b_{i+1}^{-k} l_i |x_i| |x_{i+1}| > \left( \epsilon_{i+1}b_{i+1}^{-k} |x_{i+1}|+k \right)b_i^{-k}l_i |x_i|.
\end{equation}
Since $L_{i-1}>S_{L_{i-1}}$, in order to prove inequality $(\ref{eqn:256})$, it is enough to show that
\begin{equation*}
b_{i+1}^{-k} |x_{i+1}|>\epsilon_{i+1}b_{i+1}^{-k} |x_{i+1}|+k,
\end{equation*}
which is equivalent to
\begin{equation*}
|x_{i+1}| > \frac {k b_{i+1}^k} {1-\epsilon_{i+1} }.
\end{equation*}
But $\epsilon_i<1$, so 
$$
\frac {k b_{i+1}^k} {1-\epsilon_{i+1} }<\frac {k b_{i+1}^k} {\epsilon_i-\epsilon_{i+1} }<|x_{i+1}|.
$$

To verify the second inequality we cancel the common term $b_i^{-k} l_i |x_i|$ on each side to get
\begin{equation*}
\epsilon_i b_{i+1}^{-k} |x_{i+1}| > \epsilon_{i+1}b_{i+1}^{-k} |x_{i+1}|+k,
\end{equation*}
which is equivalent to
\begin{equation*}
|x_{i+1}| > \frac {k b_{i+1}^k} {\epsilon_i-\epsilon_{i+1}},
\end{equation*}
which is given in the hypotheses.

So, we may conclude that $g_i (w,z)$ is a decreasing function of $w$ and an increasing function of $z$.  We can thus achieve an upper bound on $g_i (w,z)$ by setting $w=0$ and $z=|x_{i+1}|$:
\begin{equation*}
g_i (w,z)<g_i (0,|x_{i+1}|)=\frac {(L_{i-1}+\epsilon_i b_i^{-k}l_i|x_i|+k(l_i+1)) +|x_{i+1}|} {S_{L_i}+b_{i+1}^{-k}|x_{i+1}|}.
\end{equation*}
\end{proof}

For convenience we will define
$$
\epsilon_i'=\frac {(L_{i-1}+\epsilon_i b_i^{-k}l_i|x_i|+k(l_i+1)) +|x_{i+1}|} {S_{L_i}+b_{i+1}^{-k}|x_{i+1}|}.
$$
Thus, under the conditions of \refl{l2.4} and \refl{l2.5},
\begin{equation}
\left| \frac {N_n^Q(B,x)} {Q_n^{(k)}} -1 \right| < 2\epsilon_i'+\frac {S_n^{(k)}-Q_n^{(k)}} {S_n^{(k)}}.
\end{equation}
We will need to prove the following two lemmas in order to show that $\epsilon_i' \rightarrow 0$:

\begin{lem}\labl{l2.6} If $k \in R(W)$ then $\lim_{i \rightarrow \infty} \frac {k(l_i+1)} {b_i^{-k}l_i |x_i|}=0.$ \end{lem}
\begin{proof} 
\begin{equation*}
\frac {k(l_i+1)} {b_i^{-k}l_i |x_i|} \leq \frac {b_i^{k} 2k l_i} {l_i |x_i|}=\frac {b_i^{k} 2k} {|x_i|} \rightarrow 0
\end{equation*}
by $(\ref{eqn:good1})$. 
\end{proof}

\begin{lem}\labl{l2.7}
If $k \in R(W)$ then $\lim_{i \rightarrow \infty} \frac {\sum_{j=1}^{i-2} l_j |x_j|} {b_i^{-k} l_i |x_i|} = 0$. 
\end{lem}
\begin{proof} 
Since $\{ l_j \}$ and $\{ |x_j| \}$ are non-decreasing sequences, then
$$
\frac {\sum_{j=1}^{i-2} l_j |x_j|} {b_i^{-k} l_i |x_i|} < \frac {i l_{i-2} |x_{i-2}|} {b_i^{-k} l_i |x_i|} = \left( \frac {l_{i-2} |x_{i-2}|} {l_{i-1} |x_{i-1}|} \right) \cdot \left( ib_i^{k} \frac {l_{i-1} |x_{i-1}|} {l_i |x_i|} \right).
$$
But, by $(\ref{eqn:good2})$, $\frac {l_{i-2} |x_{i-2}|} {l_{i-1} |x_{i-1}|} \rightarrow 0$ and $ib_i^{k} \frac {l_{i-1} |x_{i-1}|} {l_i |x_i|} \rightarrow 0$. 
\end{proof}

\begin{lem}\labl{l2.8}
If $k \in R(W)$ then $\lim_{i \rightarrow \infty} \epsilon_i'=0.$
\end{lem}
\begin{proof} 
$$\epsilon_i'=\frac {\sum_{j=1}^{i-1} l_j |x_j|+\epsilon_i b_i^{-k} l_i|x_i|+|x_{i+1}|+k(l_i+1)} {\sum_{j=1}^{i-1} j^{-k} l_j |x_j|+b_i^{-k} l_i|x_i|+b_{i+1}^{-k}|x_{i+1}|} 
$$
$$
< \frac {\sum_{j=1}^{i-1} l_j |x_j|+\epsilon_i b_i^{-k} l_i|x_i|+|x_{i+1}|+k(l_i+1)} {b_i^{-k} l_i|x_i|}
$$
$$
=\frac {\sum_{j=1}^{i-2} l_j |x_j|} {b_i^{-k} l_i |x_i|}+\frac{l_{i-1}|x_{i-1}|} {b_i^{-k} l_i |x_i|}+\epsilon_i +\frac {|x_{i+1}|} {b_i^{-k} l_i |x_i|}+\frac {k(l_i+1)} {b_i^{-k} l_i |x_i|}.
$$
However, each of these terms converges to $0$ by $(\ref{eqn:good2})$, $(\ref{eqn:good3})$, \refl{l2.6} and \refl{l2.7}. \end{proof}

\section{Proof of Main Theorem 1.15}

{\bf Main Theorem 1.15} Let $W$ be a BFF and $\{x_i\}_{i=1}^{\infty}$ a $W$-good sequence. If $k \in R(W)$ then $x$ is $Q$-normal of order $k$.  If $k_i \rightarrow \infty$, then $x$ is $Q$-normal.
\begin{proof} 
Let $b$ be a positive integer, $k \in R(W)$ and let $B$ be an arbitrary block of length $k$ in base $b$.  Since $|x_i| = \omega \left(\frac {b_i^k} {\epsilon_{i-1}-\epsilon_i} \right)$, there exists $n$ large enough so that $|x_i|$ and $|x_{i+1}|$ satisfy the hypotheses of \refl{l2.5}.  Additionally, assume that $n$ is large enough so that $k \leq k_i$, $b \leq p_i$ and $S_n^{(k)}/Q_n^{(k)}<2$. Thus, by \refl{l2.4} and \refl{l2.5} 
\begin{equation} \label{eqn:T211}
\left| \frac {N_n^Q(B,x)} {Q_n^{(k)}} -1 \right|<2 \epsilon_i'+\frac {S_n^{(k)}-Q_n^{(k)}} {S_n^{(k)}}.
\end{equation}
But by \refl{l2.3}
\begin{equation} \label{eqn:T212}
\lim_{n \rightarrow \infty} \frac {S_n^{(k)}-Q_n^{(k)}} {S_n^{(k)}}=0.
\end{equation}
However, $\lim_{n \rightarrow \infty} i=\infty$. So, by \refl{l2.8}
\begin{equation} \label{eqn:T213}
\lim_{n \rightarrow \infty} \epsilon_i'=0.
\end{equation}
Thus, by $(\ref{eqn:T211})$, $(\ref{eqn:T212})$ and $(\ref{eqn:T213})$ 
$$
\lim_{n \rightarrow \infty} \left| \frac {N_n^Q(B,x)} {Q_n^{(k)}} -1 \right|=0.
$$
So,
$$
\lim_{n \rightarrow \infty} \frac {N_n^Q(B,x)} {Q_n^{(k)}} =1
$$
and we may conclude that $x$ is $Q$-normal of order $k$. 
\end{proof}

\section{Example of a $Q$-normal number for a specific $Q$}

In this section we will construct a specific example of a number that is $Q$-normal for a certain $Q$.  Recall that $C_{b,w}$ is the block in base $b$ formed by concatenating all the blocks in base $b$ of length $w$ in lexicographic order.  Since there will be $b^w$ such blocks and each is of length $w$, we arrive at
\begin{equation}\label{eqn:Cbw}
|C_{b,w}|=wb^w.
\end{equation}
We will show in \refl{l4.1} and \refl{l4.2} that $C_{b,w}$ is $(\epsilon,K,\mu)$-normal for appropriate choices of $\epsilon$, $K$ and $\mu$.  We will use this information to construct a good sequence and apply \refmt{thm3.1} to arrive at our $Q$-normal number.

\begin{lem}\labl{l4.1}
Let $n=|C_{b,w}|$.
\begin{enumerate}
\item Suppose that $1 \leq k \leq w$ and $B$ is a block of length $k$ in base $b$.  Then
\begin{equation}
(w-k+1)b^{w-k} \leq N_{n}(B,C_{b,w}) \leq wb^{w-k}.
\end{equation}
\item If $B$ is a block in base $b'>b$ and $B$ is not a block in base $b$, then $N_n(B,C_{b,w})=0$.
\end{enumerate}
\end{lem}
\begin{proof}
The second case is trivial as $C_{b,w}$ is a block in base $b$.

Suppose that $B$ is a block of length $k$ in base $b$.  Let $C_1,C_2,\ldots,C_{b^w}$ be the blocks of length $w$ in base $b$ written in lexicographic order.  Thus, $C_{b,w}=1C_11C_2 \ldots 1C_{b^w}$.  We will achieve a lower bound for $N_n(B, C_{b,w})$ by counting the number of occurrences of $B$ inside the blocks $C_i$.  In other words, we will use the estimate
\begin{equation*}
\sum_{i=1}^{b^w} N_w(B,C_i) \leq N_n(B, C_{b,w}).
\end{equation*}
For each $j$ such that $1 \leq j \leq w-k+1$, we will count the number of $i$ such that there is a copy of $B$ at position $j$ in $C_i$.  Such $j$ will correspond to copies of $B$ that don't straddle the boundary between $C_i$ and $C_{i+1}$.  Since $B$ is of length $k$ and each $C_i$ is of length $w$, there will be $w-k$ positions that are undetermined and can take on any of the values $0,1,\ldots,b-1$.  Thus, there are $b^{w-k}$ values of $i$ such that a copy of $B$ is at position $j$ of $C_i$.  Since there are $w-k+1$ choices for $j$, we arrive at the estimate
\begin{equation}
(w-k+1)b^{w-k} \leq N_n(B, C_{b,w}).
\end{equation}

In order to arrive at an upper bound for $N_n(B, C_{b,w})$, we will find an upper bound for the number of copies of $B$ that straddle the boundaries between the blocks $C_i$ and $C_{i+1}$ and add this to the number of copies of $B$ that occur inside each of the $C_i$.  These will correspond to a copy of $B$ starting at position $j$ of $C_i$ for $w-k+2 \leq j \leq w$ and finishing in $C_{i+1}$.  Given a block $D=(d_1,d_2,\ldots,d_t)$ in base $b$, define
$$
\phi(D)=d_1 b^{t-1}+d_2 b^{t-2}+\ldots+d_{t-1}b+d_t.
$$
Thus,
\begin{equation}\label{eqn:311}
\phi(C_{i+1})=\phi(C_i)+1.
\end{equation}
If a copy of $B$ is at position $j$ of $C_i$, then the first $w-j+1$ digits of $B$ are at the end of $C_i$ and the last $k-(w-j+1)$ digits of $B$ are at the beginning of $C_{i+1}$.  However, the last $w-j+1$ digits of $C_{i+1}$ are uniquely determined by $B$ from $(\ref{eqn:311})$.  The first $k-(w-j+1)$ have already directly been determined by $B$ so there are at most $w-(w-j+1)-(k-(w-j+1))=w-k$ undetermined digits of $C_{i+1}$, giving $b^{w-k}$ ways to pick $C_{i+1}$.  Additionally, there are $k-1$ positions $j$ that straddle the boundaries giving an upper bound of $(k-1)b^{w-k}$ copies of $B$ that lie on the boundaries.  Thus,
\begin{equation}
N_n(B, C_{b,w}) \leq (w-k+1)b^{w-k}+(k-1)b^{w-k}=wb^{w-k}.
\end{equation}
\end{proof}

\begin{lem}\labl{l4.2}
If $K<w$ and $\epsilon=\frac {K} {w}$, then $C_{b,w}$ is $(\epsilon,K,\lambda_b)$-normal.
\end{lem}
\begin{proof}
Let $n=|C_{b,w}|=wb^w$ and let $B$ be a block of length $k \leq K$ in base $b$. We first note that
\begin{equation}\label{eqn:321}
(w-k+1)b^{w-k}=b^{-k}n \frac {(w-k+1)b^w} {n}=\lambda_b^{(k)}(B)n \left( 1-\frac {k-1} {w} \right) > \lambda_b^{(k)}(B)n \left( 1-\frac {K} {w} \right) .
\end{equation}
We also note that
\begin{equation}\label{eqn:322}
wb^{w-k}=b^{-k}n \frac {wb^w} {n}=\lambda_b^{(k)}(B)n (1+0) < \lambda_b^{(k)}(B)n  \left( 1+\frac {K} {w} \right).
\end{equation}
Thus, by \refl{l4.1}, $(\ref{eqn:321})$ and $(\ref{eqn:322})$,
\begin{equation*}
\lambda_b^{(k)}(B)n \left( 1-\frac {K} {w} \right) < N_{n}(B,C_{b,w}) < \lambda_b^{(k)}(B)n  \left( 1+\frac {K} {w} \right).
\end{equation*}
So, $C_{b,w}$ is $(\epsilon,K,\lambda_b)$-normal. 
\end{proof}

\begin{thm}\labt{thm4.1}
Let $x_1=(0,1)$, $b_1=2$ and $l_1=0$.  For $i \geq 2$, let $x_i=C_{i,i^2}$, $b_i=i$ and $l_i=i^{3i}$. If $x$ and $Q$ are defined as in \refmt{thm3.1}, then $x$ is $Q$-normal.
\end{thm}
\begin{proof} Let $\epsilon_1=3/5$, $k_1=1$, $p_1=2$ and $\mu_1=\lambda_2$.  For $i \geq 2$, let $\epsilon_i=1/i$, $k_i=i$, $p_i=b_i$, $\mu_i=\lambda_i$ and $W=\{(l_i,b_i,p_i,\epsilon_i,k_i,\mu_i)\}_{i=1}^{\infty}$.  Thus, since $x_i=C_{b,w}$ where $b=i$ and $w=i^2$, by \refl{l4.2}, $x_i$ is $(\epsilon_i,k_i,\lambda_{b_i})$-normal.  

In order to show that $\{x_i\}$ is a $W$-good sequence we need to verify $(\ref{eqn:good1})$, $(\ref{eqn:good2})$ and $(\ref{eqn:good3})$.
Since $k_i \rightarrow \infty$, we let $k$ be an arbitrary positive integer. We will make repeated use of the fact that
\begin{equation}
|x_i|=i^2 \cdot i^{i^2}.
\end{equation}
We first verify $(\ref{eqn:good1})$:
\begin{equation}
\lim_{i \rightarrow \infty} |x_i| \Bigg/ \left( \frac {i^k} {\frac {1} {i-1}-\frac {1} {i}} \right)=\lim_{i \rightarrow \infty} \frac {i^2 \cdot i^{i^2}} {i^k \cdot i(i-1)}=\infty.
\end{equation}
We next verify $(\ref{eqn:good2})$.  Since $l_{i-1}/l_i<1$, $(i-1)^2/i^2<1$ and $(1-1/i)^{i^2}<e^{-i}$,
\begin{equation*}
\lim_{i \rightarrow \infty} \frac {\frac {l_{i-1}} {l_i} \cdot \frac {x_{i-1}} {x_i}} {i^{-1} i^{-k}} \leq \lim_{i \rightarrow \infty} i^{k+1} \cdot 1 \cdot \frac {(i-1)^2} {i^2} \cdot \frac {(i-1)^{(i-1)^2}} {i^{i^2}}
\end{equation*}
\begin{equation}
\leq \lim_{i \rightarrow \infty} i^{k+1} \cdot 1 \cdot (1-1/i)^{i^2} \cdot (i-1)^{-2i+1} \leq \lim_{i \rightarrow \infty} i^{k+1} e^{-i} (i-1)^{-2i+1}=0.
\end{equation}
Lastly, we will verify $(\ref{eqn:good3})$.  Since $(i+1)^2/i^2 \leq 2$, $(1+1/i)^{2i} < e^2$ and $(1+1/i)^{i^2}<e^i$,
\begin{equation*}
\lim_{i \rightarrow \infty} \frac {\frac {1} {l_i} \cdot \frac {|x_{i+1}|} {|x_i|}} {i^{-k}} = \lim_{i \rightarrow \infty} i^{-3i+k} \cdot \frac {(i+1)^2} {i^2} \cdot \frac {(i+1)^{(i+1)^2}} {i^{i^2}}
\end{equation*}
\begin{equation*}
\leq \lim_{i \rightarrow \infty} i^{-3i+k} \cdot 2 \cdot (1+1/i)^{i^2} \cdot (i+1)^{(2i+1)} 
\end{equation*}
\begin{equation}
\leq \lim_{i \rightarrow \infty} 2e^i (1+1/i)^{2i} i^{-i+k} (i+1) \leq \lim_{i \rightarrow \infty} 2(i+1)e^{i+2} \cdot i^{-i+k}=0.
\end{equation}
Since $\lambda_{b_i}$ is $(p_i,b_i)$-uniform, $\{x_i\}$ is a $W$-good sequence and by \refmt{thm3.1} $x$ is $Q$-normal. \end{proof}

\subsection*{Acknowledgements}
I would like to thank Vitaly Bergelson for suggesting this problem, Christian Altomare for offering many useful suggestions and Laura Harvey for her help in editing this paper.

\end{document}